\documentclass[11 pt]{article}
\usepackage{amsmath, amsthm, amssymb, amsfonts, mathrsfs}
\usepackage[normalem]{ulem} 
\usepackage{verbatim} 
\usepackage{sectsty}
\sectionfont{\centering}
\usepackage{pxfonts}
\usepackage{hyperref}
\usepackage{xcolor}
\hypersetup{
    colorlinks=true,       
    linkcolor=blue,          
    citecolor=cyan,        
}

\usepackage{longtable} 
\usepackage[center]{caption}
\usepackage{diagbox}
\usepackage{xypic}
\usepackage{mathtools}

\usepackage{tikz-cd}

\usepackage{enumitem}
\setitemize{itemsep=1pt}
\setenumerate{itemsep=1pt}

\theoremstyle{plain}
\newtheorem{theorem}{Theorem}[section]
\newtheorem{lemma}[theorem]{Lemma}
\newtheorem{corollary}[theorem]{Corollary}
\newtheorem{proposition}[theorem]{Proposition}

\newtheorem{remark}[theorem]{Remark}

\makeatletter

\newcommand{\Rmnum}[1]{\expandafter\@slowromancap\romannumeral #1@}
\makeatother

\def\ri{\mathrm i}

\def\rb{\mathbb R}
\def\nb{\mathbb N}
\def\zb{\mathbb Z}
\def\cb{{\mathbb C}}

\def\res{\mathop{\rm{Res}}}

\def\rrw{\rightarrow}
\numberwithin{equation}{section}
\allowdisplaybreaks[2] 

\title{\Large \bf Counting the number of solutions to certain infinite Diophantine equations}
\author{Nian Hong Zhou\footnote{Corresponding author} and Yalin Sun}
\date{}
\date{}

\begin{document}
\maketitle
\begin{abstract}
Let $r, v, n$ be positive integers. This paper investigate the number of solutions $s_{r,v}(n)$ of the following infinite Diophantine equations
$$
n=1^{r}\cdot |k_{1}|^{v}+2^{r}\cdot |k_{2}|^{v}+3^{r}\cdot |k_{3}|^{v}+\ldots,
$$
for ${\bf k}=(k_1,k_2,k_3,\dots)\in\zb^{\infty}$. For each $(r,v)\in\nb\times\{1,2\}$, a generating function and some asymptotic formulas of $s_{r,v}(n)$ are established.
\end{abstract}

\maketitle
\section{Introduction and statement of results}

Let $r,n$ be positive integers. A partition into $r$-th powers of an integer $n$ is a sequence of non-increasing $r$-th powers of positive integers whose sum equals $n$. Such a partition corresponds to a solution of the following infinite Diophantine equation:
\begin{equation}\label{eq1}
n=1^r\cdot k_1+2^r\cdot k_2+3^r\cdot k_3+\dots,
\end{equation}
for ${\bf k}=(k_1,k_2,k_3,\dots)\in\nb_{0}^{\infty}$. Let $p_r(n)$ be the number of partitions of $n$ into $r$-th powers and let $p_r(0):=1$, we have the generating function
\begin{equation*}
\sum_{n\ge 0}p_r(n)q^n=\prod_{n\ge 1}\frac{1}{1-q^{n^r}},
\end{equation*}
where $q\in\cb$ with $|q|<1$.\newline

Determining the values of $p_r(n)$ has a long history and can be traced back to the work of Euler.
In the famous paper \cite{MR1575586}, Hardy and Ramanujan proved an asymptotic expansion for $p_1(n)$
as $n\rrw \infty$. They \cite[p.~111]{MR1575586} also gave an asymptotic formula for $p_r(n), r\ge 2$,  without proof. In \cite[Theorem~2]{MR1555393}, Wright confirmed their asymptotic formula
$$p_{r}(n)\sim \frac{c_rn^{\frac{1}{r+1}-\frac{3}{2}}}{\sqrt{(2\pi)^{1+r}(1+1/r)}}e^{(r+1)c_rn^{\frac{1}{r+1}}},$$
as integer $n\rrw \infty$, where $c_r=\left(r^{-1}\zeta\left(1+{1}/{r}\right)\Gamma\left(1+{1}/{r}\right)\right)^{\frac{r}{r+1}}$,
$\zeta(\cdot)$ is the Riemann zeta function and $\Gamma(\cdot)$ is the classical Euler Gamma function. \newline

In this paper we investigate certain infinite Diophantine equation analogous to \eqref{eq1}. For given positive integers $n$ and $r$, we use $s_{r, v}(n)$ to denote the number of solutions of the following infinite Diophantine equation:
\begin{equation}\label{eq2}
n=1^{r}\cdot |k_{1}|^{v}+2^{r}\cdot |k_{2}|^{v}+3^{r}\cdot |k_{3}|^{v}+\ldots,
\end{equation}
for ${\bf k}=(k_1,k_2,k_3,\dots)\in\zb^{\infty}$.  The first result of this paper is about the generating function for $s_{r,v}(n)$.
\begin{proposition}\label{pro1}Let $s_{r,v}(0):=1$ and $q\in\cb$ with $|q|<1$. We have:
$$
 G_{r,1}(q):=\sum_{n\ge 0}s_{r,1}(n)q^{n}=\prod_{n\ge 1}\frac{1+q^{n^r}}{1-q^{n^r}},
$$
and
$$
G_{r,2}(q):=\sum_{n\ge 0}s_{r,2}(n)q^{n}=\prod_{j\ge 1}\prod_{n\ge 1}\frac{1-(-1)^nq^{nj^r}}{1+(-1)^nq^{nj^r}}.
$$
\end{proposition}
\begin{remark}
From the proof of this proposition (See Subsection \ref{sec21}), the above infinite product expansion for $G_{r,s}(q)~(r\in\nb, s=1,2)$ follows the identities:
$$\sum_{n\in\zb}q^{|n|}=\frac{1+q}{1-q}~~\;\mbox{and}\; ~\sum_{n\in\zb}q^{|n|^2}=\prod_{n\ge 1}\frac{1-(-q)^n}{1+(-q)^n}.$$
Their actually follow from the geometric sequence sum formula and the Jacobi triple product identity. However, any useful expansion for the sum $\sum_{n\in\zb}q^{|n|^v}$ with each integer $v>2$ is still not found yet. Therefore, whether there are infinite product formulas which is similar to Proposition \ref{pro1} for $s_{r,v}(n)~(r\in\nb, v\in\zb_{>2})$ is still a question to be settled.
\end{remark}
Thanks to the infinite product expansion in Proposition \ref{pro1}, we can determine the asymptotic behavior of $G_{r,v}(q)$ when $|q|\rrw 1^{-}$. From which we can further determine the asymptotics of $s_{r,v}(n)~\left((r,v)\in\nb\times \{1,2\}\right)$, as $n\rrw\infty$. More precisely, we prove
\begin{theorem}\label{main}For any given positive integers $r$ and $p$, we have
\begin{align*}
s_{r,1}(n)=\frac{\kappa_r^{3/2}}{\sqrt{2^{r+1}\pi^r}}\left(\frac{1}{n}\right)^{\frac{1+1/2}{1+1/r}}W_{\frac{1}{r},\frac{1}{2}}\left(\kappa_rn^{\frac{1}{1+r}}\right)\left(1+O\left(\frac{1}{n^p}\right)\right),
\end{align*}
and
\begin{align*}
s_{r,2}(n)=\frac{\kappa_r^{5/4}}{\sqrt[4]{2^{r}\pi^{r+1}}}\left(\frac{\eta(1/r)}{n}\right)^{\frac{1+1/4}{1+1/r}}W_{\frac{1}{r},\frac{1}{4}}\left(\kappa_r\eta(1/r)\left(\frac{n}{\eta(1/r)}\right)^{\frac{1}{1+r}}\right)\left(1+O\left(\frac{1}{n^p}\right)\right),
\end{align*}
as integer $n\rrw \infty$. Here $\kappa_r>0$ is given by
$$\kappa_r^{1+1/r}=2r^{-1}(1-2^{-1-1/r})\zeta(1+1/r)\Gamma(1+1/r),$$
$\eta(s)=\sum_{n\ge 1}(-1)^{n-1}n^{-s}$ is the Dirichlet eta function, and
$$W_{\alpha,\beta}\left(\lambda\right)=\frac{1}{2\pi}\int_{-1}^{1}(1+\ri u)^{\beta}\exp\left(\lambda(\alpha^{-1}(1+\ri u)^{-\alpha}+(1+\ri u) )\right)\,du,$$
for all $\alpha,\beta, \lambda>0$.
\end{theorem}
Using the standard saddle-point method, such as referring to \cite[p.127, Theorem 7.1]{MR0435697},  we can derive an asymptotic expansion for $W_{\alpha,\beta}\left(\lambda\right)$ as $\lambda\rrw +\infty$. Hence it is possible
to derive full asymptotic expansions for $s_{r,v}(n)~\left((r,v)\in\nb\times \{1,2\}\right)$. In particular, we have the following leading asymptotics:
\begin{corollary}\label{cor}
For any given positive integer $r$, we have
$$s_{r,1}(n)\sim 2^{-(r+2)/2}\pi^{-(r+1)/2}(1+1/r)^{-1/2}\kappa_rn^{-\frac{3r+1}{2+2r}}e^{(1+r)\kappa_rn^{\frac{1}{1+r}}}$$
and
$$s_{r,2}(n)\sim 2^{-(r+2)/4}\pi^{-(r+3)/4}(1+1/r)^{-1/2} \eta(1/r)^{\frac{3r}{4r+4}}\kappa_r^{3/4}n^{-\frac{5r+2}{4+4r}}e^{(1+r)\kappa_r\eta(1/r)^{\frac{r}{1+r}}n^{\frac{1}{1+r}}},$$
as $n\rrw\infty$. 
\end{corollary}

\section{Some results of the generating function}
\subsection{The proof of Proposition \ref{pro1}}\label{sec21}

We shall proceed in a formal manner to prove Proposition \ref{pro1}. Formally, using \eqref{eq2} we have
\begin{align*}
\sum_{n\ge0}s_{r,v}(n)q^{n}&=\sum_{n\ge 0}q^{n}\sum_{\substack{{\bf k}\in\zb^{\infty}\\ \sum_{j\ge 1}j^{r}|k_{j}|^{v}=n}}1\\
&=\sum_{{\bf k}\in\zb^{\infty}}q^{\sum_{j\ge 1} j^{r}| k_{j}|^{v}}
=\prod_{j\ge 1}\Bigg(\sum_{k_j\in\zb}q^{j^{r}| k_{j}|^{v}}\Bigg).
\end{align*}
Now, for $q\in\cb$ with $|q|<1$, by noting that
$$\sum_{n\in\zb}q^{|n|}=1+2\sum_{n\ge 1}q^n=\frac{1+q}{1-q}$$
and an identity of Gauss (see Andrews \cite[Corollary 2.10]{MR0557013}),
\begin{align*}
\sum_{n\in\zb}q^{n^2}=\prod_{n\ge 1}\frac{1-(-q)^n}{1+(-q)^n},
\end{align*}
we have
$$
 G_{r,1}(q):=\sum_{n\ge 0}s_{r,1}(n)q^{n}=\prod_{n\ge 1}\frac{1+q^{n^r}}{1-q^{n^r}},
$$
and
$$
G_{r,2}(q):=\sum_{n\ge 0}s_{r,2}(n)q^{n}=\prod_{j\ge 1}\prod_{n\ge 1}\frac{1-(-1)^nq^{nj^r}}{1+(-1)^nq^{nj^r}}.
$$
Clearly, the product for $G_{r,1}(q)$ is absolute convergence for all $q\in\cb$ with $|q|<1$.  For the product for $G_{r,2}(q)$, since
$$\left|\prod_{j\ge 1}\prod_{n\ge 1}\frac{1-(-1)^nq^{nj^r}}{1+(-1)^nq^{nj^r}}\right|\le \prod_{j\ge 1}\prod_{n\ge 1}\frac{1+|q|^{nj^r}}{1-|q|^{nj^r}}=
\prod_{\ell\ge 1}\left(\frac{1+|q|^{\ell}}{1-|q|^{\ell}}\right)^{\sigma_{1,r}(\ell)},$$
where
$$\sigma_{1,r}(\ell)=\#\{(n,j)\in\nb^2: nj^r= \ell\}\le \ell;$$
and hence the product is absolute convergence for all $q\in\cb$ with $|q|<1$. This completes the proof of Proposition \ref{pro1}.

\subsection{Asymptotics of the generating function}\label{sec22}
To give a proof for Theorem \ref{main}, we need to determine asymptotics of the generating function in Proposition \ref{pro1} at $q=1$.
\begin{proposition}\label{pro1a} Let $r$ be a given positive integer, $z=x+\ri y$ with $x,y\in\rb$ and $|\arg(z)|\le \pi/4$. As $z\rrw 0$,
$$G_{r,1}(e^{-z})=\frac{z^{1/2}\exp(r\kappa_r^{1+1/r}z^{-1/r})}{\sqrt{2^{r+1}\pi^r}}\left(1+O(|z|^p)\right)$$
and
$$G_{r,2}(e^{-z})=\frac{z^{1/4}\exp(r\eta(1/r)\kappa_r^{1+1/r}z^{-1/r})}{\sqrt[4]{2^{r}\pi^{r+1}}}\left(1+O(|z|^p)\right),$$
holds for any given $p>0$. Here $\kappa_r>0$ such that
$$\kappa_r^{1+1/r}=2r^{-2}(1-2^{-1-1/r})\zeta(1+1/r)\Gamma(1/r).$$
\end{proposition}
\begin{proof}
The proof of the result for $G_{r,1}(e^{-z})$ is similar to $G_{r,2}(e^{-z})$, hence we only prove the later one. We shall follow the  proof of \cite[p.89, Lemma 6.1]{MR0557013}. The series for the Riemann zeta function
$$\zeta(s)=\sum_{n\ge 1}n^{-s}$$
and the Dirichlet eta function
$$\eta(s)=\sum_{n\ge 1}(-1)^{n-1}n^{-s}$$
converge absolutely and uniformly for $s\in\cb$ when $\Re(s)\ge c>1$. Therefore, by using Mellin's transform,
\begin{align*}
\log G_{r,2}(e^{-z})&= 2\sum_{\substack{\ell\ge 1\\ \ell~ odd}}\frac{1}{\ell}\sum_{j\ge 1}\sum_{n\ge 1}(-1)^{n-1}e^{-n\ell j^rz}\\
&=2\sum_{\substack{\ell\ge 1\\ \ell~ odd}}\frac{1}{\ell}\sum_{j\ge 1}\sum_{n\ge 1}(-1)^{n-1}\frac{1}{2\pi\ri}\int_{c-\ri\infty}^{c+\ri\infty}(n\ell j^rz)^{-s}\Gamma(s)\,ds\\
&=\frac{2}{2\pi\ri}\int_{c-\ri\infty}^{c+\ri\infty}\bigg(\sum_{\substack{\ell\ge 1\\ \ell~ odd}}\frac{1}{\ell^{s+1}}\sum_{n\ge 1}\frac{(-1)^{n-1}}{n^s}\sum_{j\ge 1}\frac{1}{j^{rs}}\bigg)\Gamma(s)z^{-s}\,ds,
\end{align*}
that is
\begin{equation}\label{eqIg2}
\log G_{r,2}(e^{-z})=\frac{2}{2\pi\ri}\int_{c-\ri\infty}^{c+\ri\infty}(1-2^{-1-s})\zeta(s+1)\eta(s)\zeta(rs)\Gamma(s)z^{-s}\,ds,
\end{equation}
for all $z\in\cb$ with $\Re(z)>0$. Since the only poles of gamma function $\Gamma(s)$ are at $s=-k~(k\in\zb_{\ge 0})$,
and all are simple; $\eta(s)$ is an entire function on $\cb$; all $s=-2k~(k\in\nb)$ are zeros of zeta function $\zeta(s)$, and $s=1$ is the only pole of $\zeta(s)$ and is simple. Thus it is easy to check that the only possible poles of the integrand
\begin{equation*}
g_r(s)z^{-s}:=(1-2^{-1-s})\zeta(s+1)\eta(s)\zeta(rs)\Gamma(s)z^{-s}
\end{equation*}
are at $s=0$ and $1/r$. For all $\sigma\in[a, b], a,b\in\rb$ and real number $t, |t|\ge 1$, we have the well-known classical facts~(see \cite[p.38, p.92]{MR0364103}) that
$$\Gamma(\sigma+\ri t)\ll_{a,b} |t|^{\sigma-1/2}\exp\left(-\frac{\pi}{2}|t|\right)
\;\mbox{and}\;
\zeta(\sigma+\ri t)\ll_{a,b} |t|^{|\sigma|+1/2}.
$$
Hence we have $g_r(s)\ll_{a,b} |t|^{O(1)}\exp\left(-\frac{\pi}{2}|t|\right)$. Thus, using the residue theorem, moving the line of integration \eqref{eqIg2} to the $\Re(s)=-p$ with any given $p>0$,
and taking into account the possible pole at $s=0$ and $s=1/r$ of $g(s)$, we obtain
\begin{equation}\label{eqas}
\log G_{r,2}(e^{-z})=2\sum_{s\in\{0,1/r\}}{\rm Res}\left(g_r(s)z^{-s}\right)+O(|z|^p),
\end{equation}
as $z\rrw 0$ with $|\arg(z)|<\pi/4$. By Laurent expansion of $\zeta(s+1)$ and $\Gamma(s)$ at $s=0$, we have
$$\zeta(s+1)={1}/{s}+\gamma+O(|s|)\;\mbox{and}\; \Gamma(s)={1}/{s}-\gamma+O(|s|),$$
as $s\rrw 0$.  Therefore,
$$
\res_{s=1/r}\left(g_r(s)z^{-s}\right)=\frac{(1-2^{-1-1/r})\zeta(1/r+1)\eta(1/r)\Gamma(1/r)}{rz^{1/r}},
$$
and
$$
\res_{s=0}\left(g_r(s)z^{-s}\right)=\frac{1}{8}\log\left(\frac{z}{2^{r}\pi^{r+1}}\right).
$$
Combining \eqref{eqas} with above results, we obtain the proof of this proposition.
\end{proof}

We also need the following upper bound results.
\begin{lemma}\label{asb}Let $(r,v)\in\nb\times \{1,2\}$ be given, $z=x+\ri y$ with $x\in\rb_+$ and $y\in(-\pi,\pi]\setminus(-x,x)$. As $x\rrw 0$
$$\Re\left(\log \frac{G_{r,v}(e^{-x})}{G_{r,v}(e^{-z})}\right)\gg x^{-1/r}.$$
\end{lemma}
\begin{proof}
By using Proposition \ref{pro1} with $q\in\cb$ and $|q|<1$, we have
\begin{align*}
\log G_{r,1}(q)
&=\sum_{j\ge 1}\log\left(\frac{1+q^{j^{r}}}{1-q^{j^{r}}}\right)\\
&=\sum_{j\ge 1}\left(\sum_{\ell\ge 1}(-1)^{\ell-1}\frac{q^{\ell j^{r}}}{\ell}+\sum_{\ell\geq1}\frac{q^{\ell j^{r}}}{\ell}\right)\\
&=\sum_{\ell\ge1}\frac{1}{\ell}\sum_{j\geq1}\left((-1)^{\ell-1}+1)q^{\ell j^{r}}\right)=2\sum_{\substack{\ell\ge 1\\ \ell~odd}}\frac{1}{\ell}\sum_{j\ge 1}q^{j^{r}\ell}
\end{align*}
and
\begin{align*}
\log G_{r,2}(q)&=\sum_{n,j\ge 1}\log\left(\frac{1-(-1)^nq^{nj^r}}{1+(-1)^nq^{nj^r}}\right)\\
&=\sum_{n,j\ge 1}\sum_{\ell\ge 1}\frac{1}{\ell}\left(-(-1)^{n\ell}q^{nj^r\ell}+(-1)^{\ell}(-1)^{n\ell}q^{nj^r\ell}\right)\\
&=\sum_{\ell, j\ge 1}\frac{(-1)^{\ell}-1}{\ell}\frac{(-q^{j^r})^{\ell}}{1-(-q^{j^r})^{\ell}}=2\sum_{\substack{\ell\ge 1\\ \ell~odd}}\frac{1}{\ell}\sum_{j\ge 1}\frac{q^{j^r\ell}}{1+q^{j^r\ell}}.
\end{align*}
Furthermore,
\begin{align*}
\Re\left(\log \frac{G_{r,1}(e^{-x})}{G_{r,1}(e^{-z})}\right)&=2\sum_{\substack{\ell\ge 1\\ \ell~odd}}\frac{1}{\ell}\sum_{j\ge 1}e^{-j^r\ell x}\Re \left(1-\exp\left(2\pi\ri \ell j^r\frac{y}{2\pi}\right)\right)
\end{align*}
and
\begin{align*}
\Re\left(\log \frac{G_{r,2}(e^{-x})}{G_{r,2}(e^{-z})}\right)&=2\sum_{\substack{\ell\ge 1\\ \ell~odd}}\frac{1}{\ell}\sum_{j\ge 1}\Re\left(\frac{e^{- j^r\ell x}}{1+e^{-j^r\ell x}}-\frac{e^{- j^r\ell z}}{1+e^{-j^r\ell z}}\right)\\
&=2\sum_{\substack{\ell\ge 1\\ \ell~odd}}\frac{1}{\ell}\sum_{j\ge 1}\frac{\tanh\left(j^r\ell \frac{x}{2}\right)}{\cosh(j^r\ell x)+\cos(j^r\ell y)}\sin^2\left(j^r\ell\frac{y}{2}\right).
\end{align*}
By noting that all summand in above sums are nonnegative we have
\begin{align*}
\Re\left(\log \frac{G_{r,1}(e^{-x})}{G_{r,1}(e^{-z})}\right)
&\ge 2\sum_{j\ge 1}e^{-j^r x}\Re \left(1-\exp\left(2\pi\ri j^r\frac{y}{2\pi}\right)\right)\\
&\gg \sum_{(2\pi/x)^{1/r}< j\le 2(2\pi/x)^{1/r}}\Re \left(1-\exp\left(2\pi\ri j^r\frac{y}{2\pi}\right)\right)
\end{align*}
and
\begin{align*}
\Re\left(\log \frac{G_{r,2}(e^{-x})}{G_{r,2}(e^{-z})}\right)
&\ge 2\sum_{j\ge 1}\frac{\tanh\left(j^r \frac{x}{2}\right)}{\cosh(j^r x)+\cos(j^r y)}\sin^2\left(j^r\frac{y}{2}\right)\\
&\gg \sum_{(2\pi/x)^{1/r}< j\le 2(2\pi/x)^{1/r}}\Re \left(1-\exp\left(2\pi\ri j^r\frac{y}{2\pi}\right)\right).
\end{align*}
Thus by using the following Lemma \ref{leme} with $L=(2\pi/x)^{1/r}$, we find that
$$\Re\left(\log \frac{G_{r,v}(e^{-x})}{G_{r,v}(e^{-z})}\right)\gg \delta_r(2\pi/x)^{1/r}\gg x^{-1/r},$$
holds for all sufficiently small $x>0$ and $v\in\{1,2\}$. This finishes the proof.
\end{proof}

\begin{lemma}\label{leme}Let $r\in\nb$, $y\in\rb$ and $L\in\rb_+$ such that $L^{-r}< |y|\le 1/2$. Then there exists a constant $\delta_r\in(0,1)$ depending only on $r$ such that
\begin{equation*}
\left|\sum_{L< n\le 2L}e^{2\pi\ri n^ry}\right|\le (1-\delta_r)L,
\end{equation*}
holds for all positive sufficiently large $L$.
\end{lemma}
\begin{proof}The lemma for $r=1$ is easy and we shall focus on the cases of $r\ge 2$.
By the well-known Dirichlet's approximation theorem, for any $y\in\rb$ and $L>0$ being sufficiently large, then there exists integers $d$ and $h$ with $0<h\le L^{r-1}$ and $\gcd(h, d)=1$ such that
\begin{equation}\label{eq5}
\left|y-\frac{d}{h}\right|<\frac{1}{hL^{r-1}}.
\end{equation}
The use of \cite[Equation 20.32]{MR2061214} implies that
\begin{equation}\label{eq2100}
\sum_{L<n\le 2L}e^{2\pi \ri n^r y}=\frac{1}{h}\sum_{1\le j\le h}e^{2\pi\ri j^r\frac{d}{h}}\int_{L}^{2L}e^{2\pi\ri u^r(y-\frac{d}{h})}\,du+O(h).
\end{equation}
If real number $y$ satisfies $L^{-r}<|y|\le L^{1-r}$, then $y$ satisfies the approximation \eqref{eq5} with $(h,d)=(1,0)$. Which means that
\begin{align}\label{eq210}
\left|\sum_{L<n\le 2L}e^{2\pi \ri n^r y}\right|&=\left|\int_{L}^{2L}e^{2\pi\ri u^ry}\,du+O(1)\right|\nonumber\\
&\le 2\cdot\frac{1}{2\pi r|y|L^{r-1}}(1+2^{1-r})+O(1)\le \frac{1+2^{1-r}}{\pi r}L+O(1).
\end{align}
If real number $y$ satisfies $1/2\ge |y|\ge L^{1-r}$, then $y$ satisfies the approximation \eqref{eq5} with $h\ge 2$.
Further, by using \cite[Lemma 2.1]{MR3459558} in \eqref{eq2100}, we find that there exists a positive constant $\delta_{r1}$ depending only on $r$ such that
\begin{equation}\label{eq211}
\left|\sum_{L<n\le 2L}e^{2\pi \ri n^r y}\right|\le (1-\delta_{r1})L+O(h).
\end{equation}
On the other hand, the use of Weyl's inequality (see \cite[Lemma 20.3]{MR2061214}) implies
\begin{equation}\label{eq212}
\sum_{L< j\le 2L}e^{2\pi \ri j^r y}\ll_{\varepsilon} L^{1+\varepsilon}(h^{-1}+L^{-1}+hL^{-r})^{2^{1-r}}\ll L^{1-2^{-r-1/2}},
\end{equation}
holds for all integers $h\in( L^{1/2}, L^{r-1}]$. By using \eqref{eq210}, \eqref{eq211} and \eqref{eq212}, it is not difficult to obtain the proof of the lemma.
\end{proof}
\section{Proof of the main theorem}
From Proposition \ref{pro1a} and Lemma \ref{asb}, we can check that the sequences $\{s_{r,1}(n)\}_{n\ge 0}$ and $\{s_{r,2}(n)\}_{n\ge 0}$ satisfy the conditions of Proposition \ref{pro10} below.  Therefore, applying the following proposition, Theorem \ref{main} and Corollary \ref{cor} follows.
\begin{proposition}\label{pro10}For a given real number sequence $\{c_n\}_{n\ge 0}$ letting $G(q):=\sum_{n\ge 0}c_nq^n$.  Suppose that for $x\in \rb_+$ and $y\in(-\pi,\pi]$,
$$G(e^{-x-\ri y})-\gamma (x+\ri y)^{\beta}e^{\kappa \alpha^{-1}(x+\ri y)^{-\alpha}}\ll x^{p}G(e^{-x}),~x\rrw 0,$$
holds for any given $p>0$, where $\kappa, \gamma,\beta, \alpha\in\rb_+$. Then, for any given $p>0$ we have
$$c_n=\gamma\left(\frac{\kappa}{n}\right)^{\frac{1+\beta}{1+\alpha}}W_{\alpha,\beta}\left(\kappa^{\frac{1}{1+\alpha}}n^{\frac{\alpha}{1+\alpha}}\right)\left(1+O(n^{-p})\right),$$
as integer $n\rrw \infty$. In particular,
$$c_n\sim 2^{-1/2}\pi^{-1/2}(1+\alpha)^{-1/2}\gamma\kappa^{\frac{\beta+1/2}{1+\alpha}}n^{-\frac{1+\beta+\alpha/2}{1+\alpha}}e^{(1+\alpha^{-1})\kappa^{\frac{1}{1+\alpha}}n^{\frac{\alpha}{1+\alpha}}},~ n\rrw \infty. $$
\end{proposition}
\begin{proof}
For any given positive sufficiently large integer $n$, by using the orthogonality we have
\begin{align*}
c_n=\frac{1}{2\pi}\int_{-\pi}^{\pi}G(e^{-x-\ri y})e^{nx+n\ri y}dy.
\end{align*}
We split above integral as
\begin{align}\label{eq31}
c_n=&\frac{1}{2\pi}\left(\int_{-x }^{x}+\int_{x<|y|\le \pi}\right)\gamma (x+\ri y)^{\beta}e^{\kappa \alpha^{-1}(x+\ri y)^{-\alpha}+n(x+\ri y)}\,dy\nonumber\\
&+\frac{1}{2\pi}\int_{-\pi}^{\pi}\left(G(e^{-x-\ri y})-\gamma (x+\ri y)^{\beta}e^{\kappa \alpha^{-1}(x+\ri y)^{-\alpha}}\right)e^{n(x+\ri y)}dy\nonumber\\
=&:(I_M(n)+I_R(n))+E(n).
\end{align}
Let $x=\left(\frac{\kappa}{n}\right)^{\frac{1}{\alpha+1}}$. For $E(n)$ we estimate that
\begin{align}\label{eql1}
E(n)&\ll \int_{-\pi}^{\pi}x^{p}G(e^{-x})e^{nx}dy\nonumber\\
&\ll  \int_{-\pi}^{\pi}x^pe^{\kappa \alpha^{-1}x^{-\alpha}+nx}dy\ll n^{-p/(1+\alpha)}e^{(\alpha^{-1}+1)\kappa^{\frac{1}{1+\alpha}}n^{\frac{\alpha}{1+\alpha}}},
\end{align}
holds for any given $p>0$.  For $I_R(n)$ we estimate that
\begin{align}\label{eql1}
I_R(n)&\ll \int_{x<|y|\le \pi}e^{\kappa \alpha^{-1}\Re((x+\ri y)^{-\alpha})+nx}dy\nonumber\\
&\ll  \int_{-\pi}^{\pi}e^{\kappa \alpha^{-1}(\sqrt{2}x)^{-\alpha}+nx}dy\ll n^{-p}e^{(\alpha^{-1}+1)\kappa^{\frac{1}{1+\alpha}}n^{\frac{\alpha}{1+\alpha}}},
\end{align}
holds for any given $p>0$. For $I_M(n)$ we compute that
\begin{align*}
I_M(n)&=\frac{\gamma}{2\pi\ri }\int_{x-\ri x}^{x+\ri x}z^{\beta}e^{\kappa\alpha^{-1} z^{-\alpha}+nz}dz\\
&=\frac{\gamma x^{1+\beta}}{2\pi\ri}\int_{1-\ri}^{1+\ri}u^{\beta}e^{\kappa\alpha^{-1} x^{-\alpha} u^{-\alpha}+nxu}du\\
&=\gamma\left(\frac{\kappa}{n}\right)^{\frac{1+\beta}{1+\alpha}}\frac{1}{2\pi\ri}\int_{1-\ri}^{1+\ri}u^{\beta}e^{\kappa^{\frac{1}{1+\alpha}}n^{\frac{\alpha}{1+\alpha}}\left(\alpha^{-1}u^{-\alpha}+u\right)}du,
\end{align*}
that is
\begin{align}\label{eql2}
I_M(n)=\gamma\left(\frac{\kappa}{n}\right)^{\frac{1+\beta}{1+\alpha}}W_{\alpha,\beta}\left(\kappa^{\frac{1}{1+\alpha}}n^{\frac{\alpha}{1+\alpha}}\right).
\end{align}
By using the standard Laplace saddle-point method(see for example \cite[p.127, Theorem 7.1]{MR0435697}), since the integral
$$W_{\alpha,\beta}\left(\lambda\right)=\frac{1}{2\pi}\int_{-1}^{1}(1+\ri u)^{\beta}\exp\left(\lambda(\alpha^{-1}(1+\ri u)^{-\alpha}+(1+\ri u) )\right)\,du,$$
has a simple saddle point $u=0$, it is not difficult to prove that
\begin{align}\label{eql20}
W_{\alpha,\beta}\left(\lambda\right)\sim \frac{1}{\sqrt{2\pi(1+\alpha)}} \frac{e^{(1+\alpha^{-1})\lambda}}{\lambda^{1/2}} ,
\end{align}
as $\lambda\rrw+\infty$. The proof of Proposition \ref{pro10} follows from \eqref{eq31}--\eqref{eql2} and \eqref{eql20}. This completes the proof.
\end{proof}

\section*{Acknowledgements.}
This research was partly supported by the National Science Foundation of China (Grant No. 11971173). The authors would like to thank the anonymous referee for his/her very helpful comments and suggestions.


\begin{thebibliography}{1}

\bibitem{MR0557013}
G.~E. Andrews.
\newblock {\em The theory of partitions}.
\newblock Addison-Wesley Publishing Co., Reading, Mass.-London-Amsterdam, 1976.
\newblock Encyclopedia of Mathematics and its Applications, Vol. 2.

\bibitem{MR3459558}
A. Gafni.
\newblock Power partitions.
\newblock {\em J. Number Theory}, 163:19--42, 2016.

\bibitem{MR1575586}
G.~H. Hardy and S.~Ramanujan.
\newblock Asymptotic {F}ormulaae in {C}ombinatory {A}nalysis.
\newblock {\em Proc. London Math. Soc. (2)}, 17:75--115, 1918.

\bibitem{MR2061214}
H. Iwaniec and E. Kowalski.
\newblock {\em Analytic number theory}, volume~53 of {\em American Mathematical
  Society Colloquium Publications}.
\newblock American Mathematical Society, Providence, RI, 2004.

\bibitem{MR0435697}
F.~W.~J. Olver.
\newblock {\em Asymptotics and special functions}.
\newblock Academic Press [A subsidiary of Harcourt Brace Jovanovich,
  Publishers], New York-London, 1974.
\newblock Computer Science and Applied Mathematics.

\bibitem{MR0364103}
H. Rademacher.
\newblock {\em Topics in analytic number theory}.
\newblock Springer-Verlag, New York-Heidelberg, 1973.
\newblock Edited by E. Grosswald, J. Lehner and M. Newman, Die Grundlehren der
  mathematischen Wissenschaften, Band 169.

\bibitem{MR1555393}
E.~M. Wright.
\newblock Asymptotic partition formulae. {III}. {P}artitions into {$k$}-th
  powers.
\newblock {\em Acta Math.}, 63(1):143--191, 1934.

\end{thebibliography}

\bigskip
\noindent
{\sc Nian Hong Zhou\\
{School of Mathematics and Statistics, Guangxi Normal University\\
No.1 Yanzhong Road, Yanshan District, 541006\\
\medskip
Guilin, Guangxi, PR China}~\\
{School of Mathematical Sciences, East China Normal University\\
500 Dongchuan Road, Minhang District, 200241\\
Shanghai, PR China}}\\
Email:~\href{mailto:nianhongzhou@outlook.com; nianhongzhou@gxnu.edu.cn}{\small nianhongzhou@outlook.com; nianhongzhou@gxnu.edu.cn}

\bigskip
\noindent
{\sc Yalin Sun\\
School of Mathematical Sciences, East China Normal University\\
500 Dongchuan Road\\
Shanghai 200241, PR China}\newline
Email:~\href{mailto:yalinsun@outlook.com}{\small yalinsun@outlook.com}

\end{document}